\documentclass[leqno]{article}
\usepackage{amsmath}
\usepackage{graphicx}
\usepackage[latin1]{inputenc}
\usepackage{amssymb}
\usepackage{xcolor}
\usepackage{amsthm} 
\usepackage{enumerate}                                        
\usepackage{url}
\usepackage{graphicx}
\usepackage{epsfig}
\newtheorem{theorem}{\bf Theorem}
\newtheorem{corollary}[theorem]{\bf Corollary}

\newtheorem{proposition}[theorem]{\bf Proposition}
\newtheorem{definition}[theorem]{\bf Definition}
\newtheorem{remark}[theorem]{\bf Remark}

\numberwithin{equation}{section}
\numberwithin{theorem}{section}
\numberwithin{figure}{section}

\def\hess{\textnormal{Hess}}

\begin{document}
\renewcommand{\thefootnote}{}

\title{Half-space type theorems for a class of weighted minimal surfaces in $\mathbb{R}^{3}$} 
\author{\text{A. L. Mart\'inez-Trivi\~no$^1$, J. P. dos Santos$^2$ and G. Tinaglia$^3$}}
\vspace{.1in}
\date{}
\maketitle

\noindent \noindent  
\noindent $^1$Departamento de Matem\'atica Aplicada, Universidad de C\'ordoba, C\'ordoba, Spain \\
\noindent $^2$Departamento de Matem\'atica, Universidade de Bras\'\i lia, Bras\'\i lia, Brazil \\ 
\noindent $^3$Department of Mathematics, King's College London, London, United Kingdom \\ \\
$\text{almartinez@uco.es}^{1}$, $\text{joaopsantos@unb.br}^{2}$ and $\text{giuseppe.tinaglia@kcl.ac.uk}^{3}$ \\

\begin{abstract}

We establish half-space type results for a class of height-dependent weighted minimal surfaces in $\mathbb{R}^3$, namely critical points of a weighted area functional whose weight depends on the height. When the weight has at most quadratic growth, we prove that there are no proper surfaces contained either in two transverse vertical half-spaces of $\mathbb{R}^3$ or in a half-space determined by a non-vertical plane. We show that this second result holds in a more general context, namely, for a class of stochastically complete weighted minimal surfaces. In this setup, we also prove a result for surfaces contained in regions bounded by cones. Furthermore, for stochastically complete weighted minimal surfaces satisfying restrictions on their principal curvatures, we establish a version of the classic strong half-space result due to Hoffman-Meeks.

\end{abstract}
\vspace{0.2 cm}

\noindent 2020 {\it  Mathematics Subject Classification}: {53C42, 35J60} \\

\noindent {\it Keywords:} half-space theorems, weighted minimal surface, stochastically complete manifolds.
\everymath={\displaystyle}

\section{Introduction}

In this work, we focus on the half-space property for surfaces in $\mathbb{R}^3$ that are critical points, under normal variations with compact support, of the following weighted area functional
\begin{equation}
\label{area}
\mathcal{A}^{\varphi}(\Sigma)=\int_{\Sigma}\, e^{\varphi}\, d\Sigma
\end{equation}
on orientable surfaces $\Sigma$ in a domain $\Omega$ in Euclidean space $(\mathbb{R}^{3}, \langle \, , \, \rangle)$. The function $\varphi\vert_{\Sigma}:\Sigma\rightarrow\mathbb{R}$ is the restriction on $\Sigma$ of a smooth function $\varphi:\Omega\rightarrow\mathbb{R}$. In our case, we will consider functions $\varphi$ that depend only on one direction of $\mathbb{R}^{3}$. Without loss of generality, assume that $\varphi(x,y,z)=\varphi(z)$, i.e., $\varphi$ depends only on the variable $z$. Let $\{\vec{e}_{i}\}_{i=1,2,3}$ be the usual orthonormal frame of $\mathbb{R}^{3}$, then $\varphi\vert_{\Sigma}$ only depends on the height function $\mu:\langle p,\vec{e}_{3}\rangle:\Sigma\rightarrow \mathbb{R}$. Since $\varphi$ will only be considered when restricted to $\Sigma$, throughout this work we write $\varphi\vert_{\Sigma}=\varphi$.
\

The critical points of the area functional \eqref{area} satisfy the Euler-Lagrange equation given by the mean curvature $H$ of $\Sigma$ as follows
\begin{equation}
\label{def}
H=-\langle\overline{\nabla}\varphi, N\rangle=-\dot{\varphi}\eta,
\end{equation}
where $\overline{\nabla}$ denotes the gradient operator of $\mathbb{R}^{3}$, $(^\cdot)$ stands for the usual derivative of $\mathbb{R}$, $N$ is the Gauss map of $\Sigma$ and $\eta=\langle N,\vec{e}_{3}\rangle:\Sigma\rightarrow\mathbb{R}$ is the angle function with respect to $\vec{e}_{3}$. We say that an orientable surface $\Sigma$ in $\Omega$ is $[\varphi,\vec{e}_{3}]$-minimal if and only if its mean curvature $H$ satisfies equation \eqref{def}. 

Indeed, Ilmanen \cite{Ilm94} proved that any $[\varphi,\vec{e}_{3}]$-minimal surface in $\Omega$ is a minimal surface in $(\Omega,\langle\cdot,\cdot\rangle^{\varphi})$, with a conformally equivalent metric $$ \langle\cdot,\cdot\rangle^{\varphi}=e^{\varphi}\langle\cdot,\cdot\rangle.$$
Then, we can apply the theory of minimal surfaces in $3$-Riemannian manifolds for $[\varphi,\vec{e}_{3}]$-minimal surfaces. Unless otherwise stated, the surfaces are assumed to be connected with an empty boundary.\\

The main results of this work are the following.

 {\begin{theorem}[Theorem A]
If $\Sigma$ is a proper $[\varphi,\vec{e}_{3}]$-minimal surface in $\mathbb{R}^{3}$ whose function $\varphi$ is a diffeomorphism having at most quadratic growth and such that $\vert\dot{\varphi}\vert>\xi$, for a real constant $\xi>0$, outside a compact set of $\Sigma$, then $\Sigma$ cannot be contained in half-space $\mathcal{H}^{\varphi}_{\vec{v}}=\{x\in\mathbb{R}^{3}: \text{sgn}(\dot{\varphi})\langle x,\vec{v}\rangle\leq 0\}$ for any vector $\vec{v}\in\mathbb{R}^3$ such that $\langle\vec{v},\vec{e}_{3}\rangle>0$.
\end{theorem}}

 {\begin{theorem}[Theorem B]
If $\Sigma$ is a  proper $[\varphi,\vec{e}_{3}]$-minimal surface in $\mathbb{R}^{3}$  whose function $\varphi$ has at most quadratic growth, then $\Sigma$ cannot be contained in two transverse vertical half-spaces of $\mathbb{R}^3$.
\end{theorem}}

 The intersection of two transverse vertical half-spaces is also known as \textit{a wedge region} of $\mathbb{R}^3$. The conditions on the growth of the function $\varphi$ are natural in the theory of $[\varphi,\vec{e}_{3}]$-minimal surfaces of $\mathbb{R}^3$. For example, under these constraints, see \cite{MM,MMJ,MJ}, one proves the existence of $ [\varphi,\vec{e}_{3}]$-bowls and $[\varphi,\vec{e}_{3}]$-catenary cylinders, as well as their characterization by their asymptotic behaviours. Moreover, under such growth for $\varphi$, in \cite{MMJ} a Spruck-Xiao-type result is shown to assure convexity (i.e., $K\geq 0$, where $K$ is the Gauss curvature) for complete mean-convex $[\varphi,\vec{e}_{3}]$-minimal surfaces. On the other hand, it is established in \cite{MM} that if $\varphi$ has a growth greater than a quadratic polynomial, there exist complete $[\varphi,\vec{e}_{3}]$-bowls which are rotationally symmetric graphs contained in vertical cylinders (see Figure \ref{fig:bowl-cylinder}). Then, up to horizontal translations, Theorem B is sharp with respect to the conditions on the growth of $\varphi$. 

 \begin{figure}[htbp!]
  \centering
  \includegraphics[width=0.5\textwidth]{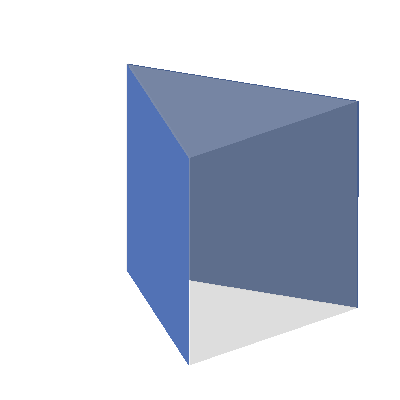}\hfill
  \includegraphics[width=0.4\textwidth]{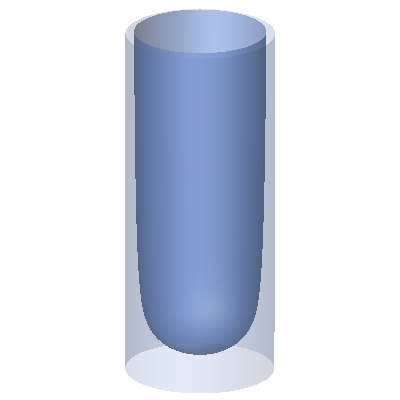}
  \caption{A \emph{wedge region} in $\mathbb{R}^3$ (left). A complete $[\varphi,\vec{e}_{3}]$-bowl contained in a vertical cylinder (right).}
  \label{fig:bowl-cylinder}
\end{figure}

Theorem A holds in a more general context, namely for $\varphi$-stochastically complete $[\varphi,\vec{e}_{3}]$-minimal surfaces. In fact, it can be shown that a proper $[\varphi,\vec{e}_{3}]$-minimal surface in $\mathbb{R}^{3}$, where $\varphi$ has at most quadratic growth is $\varphi$-stochastically complete (see Section \ref{sec-maximum-principle}). In a more general way, we can drop the properness condition for the surface and weaken the conditions on the growth of the $\varphi$ to obtain the following:

\begin{theorem}[Theorem C]
If $\Sigma$ is a $\varphi$-stochastically complete $[\varphi,\vec{e}_{3}]$-minimal surface in $\mathbb{R}^{3}$ whose function $\varphi$ is a diffeomorphism satisfying $\vert\dot{\varphi}\vert\geq\xi>0$, for some $\xi>0$, outside a compact set of $\Sigma$, then $\Sigma$ cannot be contained in half-space $\mathcal{H}^{\varphi}_{\vec{v}}=\{x\in\mathbb{R}^{3}: \text{sgn}(\dot{\varphi})\langle x,\vec{v}\rangle\leq 0\}$ for any vector $\vec{v}\in\mathbb{R}^3$ such that $\langle\vec{v},\vec{e}_{3}\rangle>0$.
\end{theorem}

 Although there are examples of $[\varphi,\vec{e}_{3}]$-minimal surfaces in $\mathbb{R}^{3}$ contained in half-spaces, as indicated by the grim-reaper cylinder, the hypothesis of stochastic completeness provides an important characterization of such surfaces. This characterization is given in the theorem below:

\begin{theorem}[Theorem D]
\label{verticalcase}
Let $\Sigma$ be a $\varphi$-stochastically complete $[\varphi,\vec{e}_{3}]$-minimal surface that is contained in a half-space of $\mathcal{H}_{\vec{v}}=\{x\in\mathbb{R}^{3}: \langle x,\vec{v}\rangle\leq 0\}$ for any vector $\vec{v}\in\mathbb{R}^3$ such that $\langle\vec{v},\vec{e}_{3}\rangle=0.$ Then, one of the following holds
\begin{enumerate}
\item $\Sigma$ is a plane parallel to $\partial \mathcal{H}_{\vec{v}}$.
\item $\Sigma$ is not a plane parallel to $\partial \mathcal{H}_{\vec{v}}$, but there exists a divergent sequence $\{p_{n}\}_{n\in\mathbb{N}}\subset\Sigma$ such that $\eta(p_{n})\rightarrow 0$ as $n\rightarrow+\infty$.
\end{enumerate}
\end{theorem}

As a combination of Theorems C and D, the well-known result of Fischer-Colbrie and Schoen \cite{FCS} for complete stable minimal surfaces in $\mathbb{R}^{3}$ and the study of non-compact properly embedded minimal surfaces in a complete Riemannian manifold with nonnegative Ricci curvature \cite[Theorem 8]{MSY}, we obtain a version of the strong half-space theorem for $[\varphi,\vec{e}_{3}]$-minimal surfaces:

\begin{theorem}[Strong half-space theorem] Let $\Sigma_{1}$ and $\Sigma_{2}$ be complete mean convex properly embedded $[\varphi,\vec{e}_{3}]$-minimal surfaces in $\mathbb{R}^{3}$ and whose function $\varphi$ is a diffeomorphism such that $\vert\dot\varphi\vert>\xi$ for a real constant $\xi>0$ outside a compact set of $\Sigma$. If $\Sigma_{1}$ and $\Sigma_{2}$ are $\varphi$-stochastically complete, then one of the following holds:
\begin{enumerate}
\item  $\Sigma_{1}$ and $\Sigma_{2}$ intersect non-trivially,
\item $\Sigma_1$ and $\Sigma_2$ are contained in distinct half-spaces determined by a vertical plane. Moreover, over each $\Sigma_{i}$, there exists a sequence of points such that one of the principal curvatures converges to zero along the sequence.
\end{enumerate}
\end{theorem}

 Finally, we establish Theorem E, which states the non-existence in a region bounded by a right circular cone (Figure \ref{fig:cone}). Theorem E extends for $[\varphi,\vec{e}_3]$-minimal surfaces the Kim-Pyo result \cite{KP}. Furthermore, the proof presented here provides a simplified proof of their result.

\begin{theorem}[Theorem E]
There are no $\varphi$-stochastically complete $[\varphi,\vec{e}_{3}]$-minimal surfaces  in a right circular cone $$C_{\vec{e}_{3},a} = \left\{ x \in \mathbb{R}^3 | \langle \dfrac{x}{||x||}, \vec{e}_{3} \rangle \leq a <1 \right\},$$ 
$a>0$, whose function $\varphi$ verifies that $0<\xi<\dot{\varphi}$ on $\Sigma$, such that $\xi>\frac{2a}{1-a^2}.$
\end{theorem} 
\begin{figure}[htbp!]
\begin{center}
\includegraphics[width=0.7\linewidth]{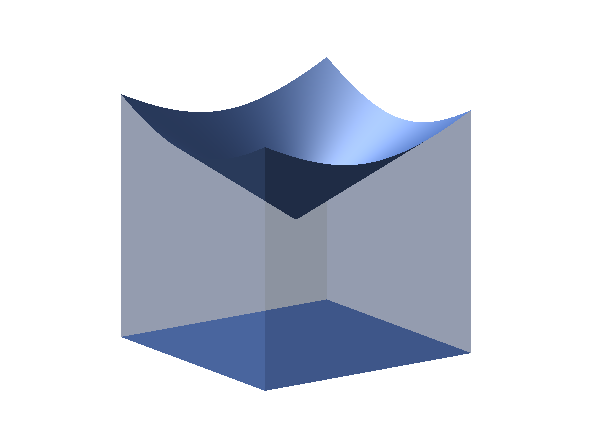}
\end{center}
\caption{Region bounded by a cone} \label{fig:cone}
\end{figure}

 The current interest in developing the theory of $[\varphi,\vec{e}_{3}]$-minimal surfaces concerns the numerous and successful advances of some particular interesting cases: 
\begin{itemize}
\item \textbf{Translators of the mean curvature flow.} When $\varphi=\text{Id}$ and $\Omega=\mathbb{R}^3$, $\Sigma$ is the initial condition of a translating soliton to the mean curvature flow, i.e., a solution to the mean curvature flow given by vertical translations $t\rightarrow\Sigma+t\vec{e}_{3}$. Clutterbuck, Schn\"urer and Schulze \cite{CSS} describe the asymptotic behaviour of rotationally symmetric examples of paraboloid type called bowl translating soliton and catenoid type called winglike. In fact, from the works of Spruck and Xiao \cite{SX}, and Wang \cite{Wang}, any complete mean convex translating soliton is convex. Therefore, the bowl translating soliton is the only entire vertical graph translating soliton. Moreover, the convexity result given in \cite{SX} was extended by Bourni, Langford and Tinaglia \cite{TLL1} to higher dimensions under rotational symmetry. Finally, in $\mathbb{R}^3$, all complete graphs have been classified thanks to the work of Hoffman, Mart\'in, Ilmanen and White \cite{HIMW} and of Wang~\cite{Wang}.

\item \textbf{Singular $\alpha$-minimal surfaces} When $\varphi(\mu)=\alpha\log(\mu)$, with $\alpha$ a real constant and $\Omega = \mathbb{R}^3_{+} = \left\{ (x,y,z) \in \mathbb{R}^3 \, , \, z>0 \right\}$, $\Sigma$ is called singular $\alpha$-minimal surface. This case has important applications in physics. For example, for $\alpha=1$, $\Sigma$ is a heavy surface under a gravitational field, which, according to the work of L\'opez \cite{RL1} and Poisson \cite{Poisson}, is of importance for the construction of perfect domes. For $\alpha=2$, we recover the classical minimal surfaces in hyperbolic space $\mathbb{H}^{3}$ in the half-space model. Moreover, the study of the global geometry of these examples was given by L\'opez \cite{RL} and Dierkes \cite{D}, whose main properties depend on the sign of $\alpha$. In fact, the geometric properties of this class of surfaces will depend on the analytical properties of the function $\varphi$.

\item \textbf{Conformal solitons to the mean curvature flow in $\mathbb{H}^3$.} When $\varphi(\mu) = 2 \log(\mu) - 1 / \mu$, $\Sigma \subset \mathbb{R}^3_+$ provides a conformal soliton to the mean curvature flow when seen as a surface in the half-space model of $\mathbb{H}^3$, with respect to the conformal vector field $-\vec{e}_3$. Such solitons were considered recently by Mari, Oliveira, Savas-Halilaj and Sodr\'e de Sena \cite{mari}, where they classified cylindrical and rotationally symmetric examples.
\end{itemize}

We recommend to the reader the works \cite{MM,MM1,MM2,MM3,MMJ,MJ} for more details on the recent advances in the theory of $[\varphi,\vec{e}_{3}]$-minimal surfaces. 

The half-space property for translating solitons is false due to the existence of a properly embedded non-flat translator in $\mathbb{R}^{3}$. For example, the  {well-known} grim reaper cylinder (Figure \ref{fig:grim-reaper}) and the $\Delta$-wing solitons  {are contained in half-spaces of $\mathbb{R}^3$}. The existence of translating solitons in slab regions was established in \cite{TLL1, TLL2, HIMW, SX, Wang}. Recently, Hoffman, Mart\'in and White \cite{HMW1} proved the existence of complete annuloid translators contained in the intersection of half-spaces determined by two parallel planes.

On the other hand, Shahriyari \cite{Sh} showed that a translator given as a complete graph cannot be defined in a cylindrical domain. Chini and Moller \cite{CM} proved that there are no properly immersed translators inside the intersection domain of two transverse half-spaces parallel to $\vec{e}_{3}$. Finally, Kim and Pyo \cite{KP} proved that there are no complete translators for the mean curvature flow in a closed half-space $\mathcal{H}_{\vec{v}}=\{x\in\mathbb{R}^{3}\, : \, \langle x,\vec{v}\rangle\leq 0\}$ with $\langle\vec{v},\vec{e}_{3}\rangle>0$. Moreover, they also proved that there are no complete translators in a right circular cone $C_{\vec{v},a}=\{x\in\mathbb{R}^{3} \, : \, \langle x,\vec{v}\rangle\leq a \vert x\vert<\vert x\vert\}$. 

\begin{figure}[htbp!]
\begin{center}
\includegraphics[width=0.5\linewidth]{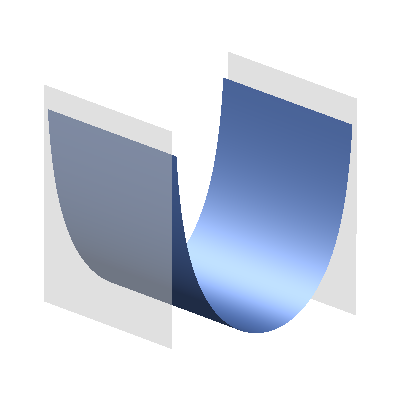}
\end{center}
\caption{The grim reaper cylinder and its asymptotic behaviour between two planes} \label{fig:grim-reaper}
\end{figure}

Since any translator is minimal in its corresponding Ilmanen space, one can consider the strategy of Hoffman and Meeks \cite{HM}, collapsing the neck of catenoids and then applying the maximum principle of tangency. However, the behaviour of catenoid-type translators is different. If we consider the family of winglike collapsing necks, these converge to a double cover of the bowl translator.

For this reason, the arguments must be different. We will use a more general maximum principle called the weak maximum principle for $\varphi$-stochastically complete surfaces together with a generalized Omori Yau maximum principle, see \cite{AMR}, for the Schr\"odinger operator drift Laplacian $\Delta^{\varphi}(\cdot)=\Delta (\cdot)+\langle\nabla\varphi,\nabla (\cdot)\rangle$, where $\nabla$ is the induced gradient on $\Sigma$. We would like to point out that the Omori-Yau maximum principle for $\Delta^{\varphi}$ holds for translators in $\mathbb{R}^{3}$ that are proper. This technique was used, for example, by Impera and Rimoldi \cite{IR} to obtain half-space type results for $\varphi$-stochastically complete translators. 
\\

The paper is organized as follows: Section 2 is an introduction to the weak maximum principle for stochastically complete surfaces. Moreover, we give conditions over the mean curvature where the Omori-Yau maximum principle holds for the drift Laplacian $\Delta^\varphi$, and we establish sufficient conditions to guarantee that the Omori-Yau maximum principle holds on complete $[\varphi,\vec{e}_3]$-minimal surfaces with boundary. In Sections 3 and 5, we prove Theorems A and C and Theorem D, respectively. Precisely, the Half-space properties for $\varphi$-stochastically complete $[\varphi,\vec{e}_{3}]$-minimal surfaces in $\mathbb{R}^3$. As a consequence of Theorems C and D, we have a strong half-space type result for this family of surfaces. In Section 4, we show Theorem B, where we obtain the bi-halfspace theorem determined by wedges of $\mathbb{R}^3$.
Finally, in Section 6, we present a non-existence result for bounded regions by a right circular cone, proving Theorem E, which extends the Kim-Pyo result \cite{KP} with a simplified proof.

\section{Brief introduction to weak maximum principle} \label{sec-maximum-principle}
All details of the following definitions and results can be found in \cite{AMR, PRS}.

\begin{definition}{\rm 
 A Riemannian manifold $(\Sigma,\langle\cdot,\cdot\rangle)$ is said to be \textit{stochastically complete} if the \textit{weak maximum principle} holds for the Laplacian operator $\Delta$, that is, for every $u\in C^2(\Sigma)$ with $u^{*}=\sup_{\Sigma}\, u<+\infty$ and any $\rho<u^{*}$, then $$\inf_{\Sigma_{\rho}}\Delta u\leq 0, \ \ \text{where} \ \ \Sigma_{\rho}=\{x\in\Sigma: u(x)>\rho\}. $$
In other words, for any function $u\in C^{2}(\Sigma)$ with $u^{*}=\sup_{\Sigma}u<+\infty$, there exists a sequence of points $\{p_{n}\}_{n\in\mathbb{N}}\subset\Sigma$ satisfying
$$u(p_{n})>u^{*}-\frac{1}{n}, \ \ \text{and} \ \ \Delta u(p_{n})<\frac{1}{n}, \text{ for any } n\in\mathbb{N}.$$
} \label{def:sthocastically-complete}
\end{definition}
It is important to note that, as pointed out in \cite{AMR}, the Riemannian manifold $\Sigma$ is not assumed to be geodesically complete (see the observation just after \cite[Definition 2.2]{AMR}). 

In this paper, we shall apply the weak maximum principle for the drift Laplacian operator $$\Delta^{f}(\cdot)=\Delta(\cdot)+\langle\nabla f,\nabla (\cdot)\rangle,$$
for some $f\in C^{2}(\Sigma)$, where $\nabla$ denotes the gradient operator in $\Sigma$. 
\begin{definition}{\rm 
Let $(\Sigma,\langle\cdot,\cdot\rangle)$ be a Riemannian manifold and consider $f\in C^{2}(\Sigma)$. We say that $\Sigma$ is $f$-\textit{stochastically complete} if the weak maximum principle holds for $\Delta^{f}$, that is,  if for any function $u\in C^{2}(\Sigma)$ with $u^{*}=\sup_{\Sigma}u<+\infty$, there exists a sequence of points $\{p_{n}\}_{n\in\mathbb{N}}\subset\Sigma$ satisfying
$$u(p_{n})>u^{*}-\frac{1}{n}, \ \ \text{and} \ \ \Delta^{f} u(p_{n})<\frac{1}{n}, \text{ for any } n\in\mathbb{N}.$$}
\end{definition}
Moreover, from \cite[Theorem 2.14]{AMR}, the following characterization holds,
\begin{proposition}
The following statements are equivalent
\begin{itemize}
\item $\Sigma$ is $f$-stochastically complete.
\item For every $\lambda>0$, the only non-negative bounded $C^{2}$ solution of $\Delta^{f} u\geq\lambda u$ on $\Sigma$ is the constant zero.
\item For every $\lambda>0$, the only nonnegative bounded $C^{2}$ solution of $\Delta^{f} u=\lambda u$ on $\Sigma$ is the constant zero.
\end{itemize}
\end{proposition}
Notice that, in general, it is difficult to check when a Riemannian manifold verifies the $f$-weak maximum principle. To end this part, we give a criterion for verifying this property when $u^{*}>0$, together with an explicit example. From Theorem \cite[Theorem 2.14]{AMR}, we can adapt the arguments of \cite[Theorem 2.9]{AMR} to prove the following result,
\begin{theorem}
\label{criteria}
Let  $\Sigma$ be a Riemannian manifold and $f\in C^{2}(\Sigma)$. If there exists a function $\gamma\in C^{\infty}(\Sigma)$ such that $\gamma(p)\rightarrow +\infty$ as $p\rightarrow \infty$ and $\Delta^{f}\gamma\leq\, \lambda\gamma$ outside a compact set, for some $\lambda>0$, then the weak maximum principle holds in $\Sigma$ for any function $u\in C^{2}(\Sigma)$ with $0<u^{*}<+\infty$.
\end{theorem}
\begin{definition}{\rm 
Let $(\Sigma,\langle\cdot,\cdot\rangle)$ be a Riemannian manifold and consider $f\in C^{2}(\Sigma)$. We say that the Omori-Yau maximum principle holds for the drift Laplacian $\Delta^f$ on $\Sigma$ if for any $u\in(\Sigma)$ with $u^{*}=\sup_{\Sigma}u<+\infty$, there exists a sequence of points $\{p_{n}\}_{n\in\mathbb{N}}\subset\Sigma$ satisfying
$$
\textnormal{i)}\, u(p_{n})>u^{*}-\frac{1}{n}, \ \ \textnormal{ii)}\, \vert\nabla u (p_{n}) \vert<\frac{1}{n} \,\, \textnormal{ and } \ \ \textnormal{iii)}\, \Delta^{f} u(p_{n})<\frac{1}{n},
$$
 for any  $n\in\mathbb{N}.$}
\end{definition}
\begin{remark}
If the Omori-Yau maximum principle holds in $\Sigma$ for $\Delta^{f}$, see  \cite[Theorem 3.2]{AMR}, then $\Sigma$ is $f$-stochastically complete.
\end{remark}
In the sequence, we give a family of $[\varphi,\vec{e}_{3}]$-minimal surfaces where these maximum principles hold for the drift Laplacian $\Delta^{\varphi}$. For this,  we must introduce the following result \cite[Theorem 7]{PRS}.
\begin{theorem} \label{thm:OY-G}
Let $\Sigma$ be a Riemannian manifold and consider $f\in C^2(\Sigma)$. Assume that there exists a function $\gamma\in C^2(\Sigma)$ such that the following properties hold outside a compact set
\begin{align}
\label{KM}
&\gamma(p)\rightarrow+\infty \ \ \text{as} \ \ p\rightarrow\infty\nonumber \\
& \vert\nabla\gamma\vert\leq A\sqrt{\gamma} \\
& \Delta^{f}\gamma\leq B\sqrt{\gamma\mathcal{G}(\sqrt{\gamma})}\nonumber,
\end{align}
for some $A,B>0$ and $\mathcal{G}:[0,+\infty[\rightarrow\mathbb{R}$ a smooth function satisfying $\mathcal{G}(0)>0$, $\mathcal{G}'(t)\geq 0$ on $[0,+\infty[$, $\mathcal{G}^{-1/2}$ is not integrable and
$$\limsup_{t\rightarrow+\infty}\frac{t\mathcal{G}(\sqrt{t})}{\mathcal{G}(t)}<+\infty.$$
Then, the Omori-Yau maximum principle holds for the drift Laplacian $\Delta^{f}$.
\end{theorem}
A special example verifying the previous properties is the following function
$$\mathcal{G}(t)=t^2\Pi_{j=1}^{n}\left(\log ^{(j)}(t)\right)^2 \ \ \text{for any} \ \ t>>1,$$
where $\log^{(j)}$ is the $j$-iterated logarithm. Hence, from \cite[Example 7]{PRS}, if $\Sigma$ is an isometric proper surface in $\mathbb{R}^3$ such that the cut locus of $0$ is disjoint to $\Sigma$ and the mean curvature $H$ of $\Sigma$ satisfies
$$\vert H(p)\vert+\vert\nabla f\vert\leq B\sqrt{\mathcal{G}(\vert p\vert)},$$
for some constant $B>0$ in the complement of a compact set of $\Sigma$, then the Omori-Yau maximum principle holds for $\Delta^{\varphi}.$ Finally, as a direct consequence of this result, we can give the following application for the family of isometric proper $[\varphi,\vec{e}_3]$-minimal surfaces in $\mathbb{R}^3$ where $\varphi$ satisfies the following condition outside a compact set of $\Sigma$
\begin{equation}
\label{condK}
\vert\dot{\varphi}(\mu(p))\vert \leq B\sqrt{\mathcal{G}(\vert p\vert)}, \ \ \text{ for some constant } B>0.
\end{equation}
\begin{remark}
Notice that, up to horizontal translation, we can always assume that the cut locus of $0$ is disjoint from $\Sigma$.
\end{remark}
\begin{corollary}
The Omori-Yau maximum principle for $\Delta^{\varphi}$ holds on any proper $[\varphi,\vec{e}_{3}]$-minimal surface in $\mathbb{R}^{3}$ with function $\varphi$ satisfying \eqref{condK}. In particular, the weak maximum principle for $\Delta^{\varphi}$ holds.
\end{corollary}
In the particular case $\mathcal{G}(t)=t^2$, we have the well-known Khas'minskii test for any gradient Schr\"ondinger operator on $\Sigma$. The reader can see all details in \cite{AMR}. For instance, an application in this case is given in \cite{MMJ}, where the Omori-Yau maximum principle is applied for properly embedded $[\varphi,\vec{e_{3}}]$-minimal surfaces in $\mathbb{R}^3$. In our case, we have the following corollaries.

\begin{corollary} \label{cor:main}
The Omori-Yau maximum principle for $\Delta^{\varphi}$ holds on any isometric proper $[\varphi,\vec{e}_{3}]$-minimal surface in $\mathbb{R}^{3}$ whose function $\varphi$ has at most quadratic growth. In particular, the weak maximum principle for $\Delta^{\varphi}$ holds.
\end{corollary}
\begin{corollary}
\label{cor:main2}
The Omori-Yau maximum principle for $\Delta^{\varphi}$ holds for any properly immersed surface with bounded mean curvature. In particular, it holds for minimal surfaces, translators to the mean curvature flow, and any $\alpha$-singular minimal surface of $\mathbb{R}^3$, with vertical height bounded away from zero.
\end{corollary}

We may wonder when we have an Omori-Yau maximum principle for $\Delta^{\varphi}$ on complete surfaces with boundary in $\mathbb{R}^3$. In this context, we conclude this section by providing sufficient conditions for applying this maximum principle. Following \cite{CM} together with \cite{X}, we have the following result

\begin{theorem}[Omori-Yau maximum principle with boundary]
\label{OYboundary}
Let $\Sigma$ be a complete proper $[\varphi,\vec{e}_{3}]$-minimal surface in $\mathbb{R}^3$ with boundary $\partial\Sigma$ ({possibly empty}) and whose function $\varphi$ has at most quadratic growth. Suppose that $f:\Sigma\rightarrow\mathbb{R}$ is a function satisfying the following it
\begin{enumerate}
    \item $\sup_{\Sigma}\vert f\vert<+\infty \ \ , \ \ \sup_{\partial\Sigma} f<\sup_{\Sigma}f.$
    \item $f$ is continuous on $\Sigma$.
    \item There exists $\epsilon_{f}>0$ such that $f$ is the class $C^2$ on the subset
    $$\{p\in\Sigma : f(p)>\sup_{\Sigma}f-\epsilon_{f}\}.$$
\end{enumerate}
Then, there exists a sequence $\{p_{n}\}_{n\in\mathbb{N}}\subset\Sigma$ such that
$$f(p_{n})\rightarrow \sup_{\Sigma}f \ \ , \vert\nabla f (p_{n})\vert\rightarrow 0 \ \ \text{and} \ \ \Delta^{\varphi}f (p_{n})\leq 0.$$
\end{theorem}
\begin{proof}
Consider the function $r^2:\Sigma\rightarrow\mathbb{R}$ given by $r^2 (p)=\langle p,p\rangle$ for any $p\in\Sigma$. If there exists a smooth positive function $\psi:\Sigma\rightarrow ]0,+\infty[$ such that  $\psi\rightarrow +\infty \text{ as } r\rightarrow+\infty,$ $\vert\nabla\psi\vert\leq \delta$ and $\vert\Delta^{\varphi}\psi\vert\leq \delta,$ for some $\delta>0$. Then, we can adapt the proof of Xin \cite{X} in our situation since $f$ is bounded on $\Sigma$. 

Define $\psi_{a}:\Sigma\rightarrow\mathbb{R}$ given by $\psi_{a}(p)=\text{log}(a+r^2(p))$ for any $p\in\Sigma$. Taking $a>0$ large enough, we can assume that $\psi$ is smooth everywhere and $\psi>0$ on $\Sigma$. Moreover, by a straightforward computation, we get the following
$$\vert\nabla\psi_{a}(p)\vert=\frac{\nabla r^2(p)}{a+r^2(p)}=\frac{2 p^T}{a+r^2(p)},$$
$$\Delta^{\varphi}\psi_{a}(p)=-\frac{\vert\nabla r^2(p)\vert^2}{(a+r^2(p))^2}+\frac{4+2\dot{\varphi}(\mu)\mu}{a+r^2(p)}.$$
Hence, the proof follows since $\varphi$ has at most quadratic growth.
\end{proof}

Finally, we present a proposition, which is a consequence of Ekeland principle, that states that for any $C^1$ function $u$ on a complete manifold $\Sigma$ with  $u^* = \sup_{\Sigma} u < + \infty$, we can find a sequence $\left\{ x_n \right\} \subset \Sigma$ satisfying i) and ii) of Definition 2.5 (see \cite[Proposition 2.2]{AMR}). Namely,

\begin{proposition} \label{ekeland}
    Let $(\Sigma,\langle\cdot,\cdot\rangle)$ be a complete manifold and $u : \Sigma \rightarrow \mathbb{R}$ a $C^1$ function such that $u^* = \sup_{\Sigma} u < + \infty$. Then, for every sequence $\left\{ y_n \right\} \subset \Sigma$ such that $u(y_n) \rightarrow u^*$ as $n \rightarrow + \infty$ there exists a sequence $\left\{ x_n \right\}$ with the properties
  $$ \text{ i) } u(x_n) \rightarrow u^* , \ \ \text{ ii) }|\nabla u (x_n) | \rightarrow 0 \ \ \text{and} \ \ \text{iii) } d(x_n,y_n) \rightarrow 0
    $$
    as $n \rightarrow + \infty$.
\end{proposition}

\section{Proofs of Theorems A and C}
We first prove Theorem C, as Theorem A will be a consequence of Theorem C and Corollary \ref{cor:main}.

\begin{proof}[Proof of Theorem C] We argue by contradiction.
Suppose that there exists a $\varphi$-stochastically complete $[\varphi,\vec{e}_{3}]$-minimal surface in the half-space $$\mathcal{H}^{\varphi}_{\vec{v}}=\{x\in\mathbb{R}^{3}: \text{sgn}(\dot{\varphi})\langle x,\vec{v}\rangle\leq 0\},$$ with $\langle \vec{v},\vec{e}_{3}\rangle>0$ and
$$\text{sgn}(\dot{\varphi})=\left\lbrace\begin{array}{c}
 1 \ \ \text{ if } \dot{\varphi} >0, \\
-1 \ \ \text{ if } \dot{\varphi}<0.
\end{array}\right.$$ Without loss of generality, we can assume that the vector field $\vec{v}=\vec{e}_{3}+\vec{a}$ is such that $\langle \vec{v},\vec{e}_{3}\rangle=c>0$, where $\vec{a}$ is a constant vector, and $c$ is a positive real number. Now, consider the height function of $\Sigma$ with respect to the direction $\vec{v}$ given by $\mu_{\vec{v}}:\Sigma\rightarrow\mathbb{R}$, given by $\mu_{\vec{v}}(p)=\langle p,\vec{v}\rangle$ for $p\in\Sigma$. 

Fix any point $p\in\Sigma$. For any vector field $X\in T_{p}\Sigma$, we have
$$\langle X,\nabla\mu_{\vec{v}}\rangle=d\mu_{\vec{v}}(p)(X)=\langle X,\vec{v}^{T}\rangle,$$
where $(\cdot)^T$ denotes the projection on the tangent bundle of $\Sigma$. Then, if we denote by $\{v_{i}\}$, $i \in \left\{1,2 \right\}$, an orthonormal frame of $T_{p}\Sigma$, by $D$ the Levi-Civita connection of $\mathbb{R}^{3}$ and by $N$ the unit normal vector of $\Sigma$, we obtain that
\begin{equation}
\label{lap}
\langle D_{v_{i}}\nabla\mu_{\vec{v}},v_{i}\rangle=\langle D_{v_{i}}\vec{v}^{T},v_{i}\rangle=-\langle\vec{v},N\rangle\langle D_{v_{i}}N,v_{i}\rangle=-\langle\vec{v},N\rangle\mathcal{S}(v_{i},v_{i}),
\end{equation}
where $\mathcal{S}$ stands the second fundamental form of $\Sigma$. Hence, from equation \eqref{lap} together with the definition of being $[\varphi,\vec{e}_{3}]$-minimal, we get the
\begin{equation}
\label{lapmu}
\Delta\mu_{\vec{v}}=\dot{\varphi}\eta\langle \vec{v},N \rangle.
\end{equation}
Moreover, since $\varphi$ restricted to $\Sigma$ is given by $\varphi (p)  = \varphi( \langle p, \vec{e}_3 \rangle )$, for $p \in \Sigma$, we have
\begin{equation}
\label{grad}
\langle\nabla\varphi,\nabla\mu_{\vec{v}}\rangle=\dot{\varphi}\langle\vec{e}_{3}^{T},\vec{v}^{T}\rangle=\dot{\varphi}\left(\langle\vec{e}_{3},\vec{v}\rangle-\langle N,\vec{e}_{3}\rangle\langle N,\vec{v}\rangle \right).
\end{equation}
From \eqref{lapmu} and \eqref{grad}, we finally obtain that
\begin{equation}
\label{driftmu}
\Delta^{\varphi}\mu_{\vec{v}}=\dot{\varphi}\langle\vec{v},\vec{e}_{3}\rangle=c\dot{\varphi}.
\end{equation}
Assume, without loss of generality, that $\text{sgn}(\dot{\varphi})=1$ (the same reasoning works for the other case). In this case, the contradiction hypothesis implies that there exists $\mu^{*}=\text{sup}_{\Sigma}\mu_{\vec{v}}$. Firstly, notice that this supremum is not achieved in an interior point from \eqref{driftmu}, once $c \dot{\varphi}>0$. On the other hand, since $\Sigma$ is $\varphi$-stochastically complete, there is a divergent sequence of points $\{p_{n}\}_{n\in\mathbb{N}}\subset\Sigma$ such that
$$\mu_{\vec{v}}(p_{n})\rightarrow\mu^{*} \ \ \text{and} \ \ \Delta^{\varphi}\mu_{\vec{v}}(p_{n})\leq \frac{1}{n} \ \ \text{ for any } n\in\mathbb{N}.$$
However, taking the limit in \eqref{driftmu}, we get a \textit{contradiction} since
$$ \Delta^{\varphi}\mu_{\vec{v}} (p)\geq c\, \xi>0 \text{ on } \Sigma.$$
\end{proof}
As a consequence of Theorem C, if we replace the condition of being $\varphi$-stochastically complete by a hypothesis on the growth of $\varphi$ as we have in Corollary \ref{cor:main}, we obtain the following result. 

\begin{corollary}
\label{Cor}
A $[\varphi,\vec{e}_{3}]$-minimal surface $\Sigma$ satisfying the generalized Omori-Yau maximum principle for $\Delta^{\varphi}$, whose function $\varphi$ is a diffeomorphism such that $\vert\dot{\varphi}\vert\geq\xi>0$ outside a compact set of $\Sigma$, cannot be contained in the closed half-space $\mathcal{H}^{\varphi}_{\vec{v}}$ with $\langle\vec{v},\vec{e}_{3}\rangle>0$. In particular, if $\Sigma$ is a proper $[\varphi,\vec{e}_{3}]$-minimal surface in $\mathbb{R}^3$ whose function $\varphi$ is a diffeomorphism satisfying $0<\xi \leq \vert\dot{\varphi}(\mu(p))\vert \leq B\sqrt{\mathcal{G}(\vert p\vert)}$ outside a compact set of $\Sigma$, for real constants $\xi$ and $B$ and a smooth function $\mathcal{G}$ satisfying the conditions of Theorem \ref{thm:OY-G}, then $\Sigma$ cannot be contained in the closed half-space $\mathcal{H}^{\varphi}_{\vec{v}}$ with $\langle\vec{v},\vec{e}_{3}\rangle>0$.
\end{corollary}
As stated previously, Theorem A is then a consequence of Theorem C and Corollary \ref{cor:main}, in the sense of the previous corollary.

\begin{proof}[Proof of Theorem A]
By Corollary \ref{cor:main}, the generalized maximum principle for $\Delta^{\varphi}$ holds in $\Sigma$. Then, the rest of the proof follows from Corollary \ref{Cor}.
\end{proof}

The following corollary shows that we can allow $\dot{\varphi}$ to be unbounded from above as long as we keep a bound for $H$. In this case, if $\dot{\varphi}$ diverges along with a sequence $p_n$, then we necessarily have $\eta(p_n) \to 0$.

\begin{corollary}
\label{Cor32}
Let $\Sigma$ be a mean convex proper $[\varphi,\vec{e}_{3}]$-minimal surface in $\mathbb{R}^{3}$ with mean curvature $H$ satisfying $0<\xi_1 <  H  < \xi_2$ on $\Sigma$, for some constants $\xi_i>0,$ $i=1,2$. Then $\Sigma$ cannot be contained in the closed half-space $\mathcal{H}_{\vec{v}}^{\eta}=\{x\in\mathbb{R}^{3}:\text{sgn}(\eta)\langle x,\vec{v}\rangle\geq 0\}$ with $\langle \vec{v},\vec{e}_{3}\rangle>0$.
\end{corollary}

\begin{proof}
The conditions on $H$ ensure that $\Sigma$ satisfies the generalized Omori-Yau maximum principle for $\Delta^{\varphi}$. Furthermore, they imply that the derivative of $|\dot{\varphi}|>\xi_1$ on $\Sigma$. The result follows from Corollary \ref{Cor}.
\end{proof}

\section{Proof of Theorem B}

The proof closely follows the steps for the proof of Theorem 1.1 in \cite{CM} together with the Omori-Yau maximum principle with boundary \ref{OYboundary}. To maintain self-containment of the paper, we will describe the basic concepts and notation here. Without loss of generality, we can consider the vertical planes that determine the wedge passing through the $z$-axis. In this case, for $w_1 = (a,b,0)$ and $w_2=(a,-b,0)$, $a>0$, we define the half-space determined by $w_i$ and its boundary given by a plane $P_i$ as
$$
\begin{array}{rcl}
H_i &:=& \left\{ x \in \mathbb{R}^3\,|\, \langle x, w_i \rangle \geq 0 \right\} \\
P_i &:=& \partial H_i  = \left\{ x \in \mathbb{R}^3\,|\, \langle x, w_i \rangle = 0 \right\}. \\
\end{array}
$$
Therefore, the wedge that we are going to consider will be the intersection $H_1 \cap H_2$.

In what follows, we will construct a function in which we will apply the maximum principle. Let $P_i + R w_i$, $R>0$ the parallel displacement of $P_i$ in the direction of $w_i$. A straightforward computation shows that the intersection of such planes is given by the line $\mathcal{L}_R$ as 
$$
(P_1 + R w_1) \cap (P_2 + R w_2) = \mathcal{L}_R := \left\{ \left( \dfrac{R}{a},0,x_3 \right), x_3 \in \mathbb{R} \right\}.
$$
and the Euclidean distance of a point $x=(x_1,x_2,x_3)$ to $\mathcal{L}_R$ is given by 
$$d_R(x) = \sqrt{\left(x_1 - \dfrac{R}{a}\right)^2+x_2^2}.$$ 

Consider now the cylindrical set
\begin{equation}
\mathcal{D}_R = \left\{ x \in \mathbb{R}^3\,|\, d_R(x) \leq R \right\}.
\end{equation}
We observe that $\partial \mathcal{D}_R$ is the right circular cylinder contained in $W$ and tangent to $P_1$ and $P_2$. Let $\mathcal{V}_R$ be the connected component of $W \setminus \mathcal{D}_R$ such that $d_R$ is bounded. A second observation is that $\mathcal{V}_R$ fills in $W$ as $R$ goes to infinity.

Suppose by contradiction that $\Sigma \subset W$. If $\Sigma \cup \mathcal{V}_R \neq \varnothing$, we consider the function $f : \Sigma \rightarrow \mathbb{R}$
\begin{equation}
f(p) = \left\{
\begin{array}{ll}
d_R(p),\,\, p \in \Sigma \cap \mathcal{V}_R \\
R,\,\,\textnormal{elsewhere in}\,\, \Sigma
\end{array}
\right.
\end{equation}

To compute the gradient and Laplacian of $f$, it is enough to analyse the function $d_R$ restricted to $\Sigma$. Let $D^{\mathbb{R}^3}$ be the Euclidean connection. Since the Euclidean gradient $\nabla^{\mathbb{R}^3}$ of $d_R$ is given as
\begin{equation}
\label{gradEu}
\nabla^{\mathbb{R}^3}d_R = \dfrac{1}{d_R} \left( x_1 - \dfrac{R}{a}, x_2, 0\right) = \dfrac{\left( x_1 - \frac{R}{a} \right)}{d_R} e_1 + \dfrac{x_2}{d_R} e_2,
\end{equation}
we conclude that
$$
\begin{array}{rcl}
D^{\mathbb{R}^3}_{e_1} \left( \nabla^{\mathbb{R}^3}d_R \right) &=& -\dfrac{\left(x_1 - \frac{R}{a} \right)}{d^2_R} \nabla^{\mathbb{R}^3}d_R + \dfrac{1}{d}e_1, \\
D^{\mathbb{R}^3}_{e_2} \left( \nabla^{\mathbb{R}^3}d_R \right) &=& -\dfrac{x_2}{d^2_R} \nabla^{\mathbb{R}^3}d_R + \dfrac{1}{d}e_2, \\
D^{\mathbb{R}^3}_{e_3} \left( \nabla^{\mathbb{R}^3}d_R \right) &=& 0.
\end{array}
$$
Therefore,
$$
\begin{array}{rcl}
\hess^{\mathbb{R}^3} d_R (e_1,e_1) &=& - \dfrac{(x_1 - \frac{R}{a})^2}{d_R^3} + \dfrac{1}{d_R}, \\
\hess^{\mathbb{R}^3} d_R (e_1,e_2) &=& - \dfrac{x_2^2}{d_R^3} + \dfrac{1}{d_R}, \\
\hess^{\mathbb{R}^3} d_R (e_1,e_2) &=& - \dfrac{x_2(x_1 - \frac{R}{a})^2}{d_R^3}, \\
\hess^{\mathbb{R}^3} d_R (e_i,e_3) &=& 0. \\
\end{array}
$$
Consequently, we have $ \hess^{\mathbb{R}^3} d_R (\nabla^{\mathbb{R}^3}d_R,\nabla^{\mathbb{R}^3}d_R)=0$ and $\Delta^{\mathbb{R}^3}d_R = \dfrac{1}{d_R}$.

Let $V$ be the vector field $V = \dfrac{1}{d_R} \left( - \dfrac{x_2}{a}, x_1 - \dfrac{R}{a}, 0\right),$ obtained as a $\frac{\pi}{2}$-rotation of $\nabla^{\mathbb{R}^3}d_R$. Then $\left\{ \nabla^{\mathbb{R}^3}d_R, V, e_3 \right\}$, is an orthonormal frame in $\mathbb{R}^3$. A straightforward computation shows that
$$
\begin{array}{rcl}
\hess^{\mathbb{R}^3} d_R (V,V) &=& \dfrac{1}{d_R}, \\
\hess^{\mathbb{R}^3} d_R (V,\nabla^{\mathbb{R}^3}d_R) &=& 0.
\end{array}
$$

Now we restrict the function $d_R$ to $\Sigma$ to obtain the corresponding differential operators for $f$. Let $N$ be the unit normal vector field of $\Sigma$. We can decompose $N$ as $N = \langle N, V \rangle V + \langle N ,  \nabla^{\mathbb{R}^3}d_R \rangle \nabla^{\mathbb{R}^3}d_R + \langle N , e_3 \rangle e_3.$ Therefore,
$$
\hess^{\mathbb{R}^3} d_R (N,N) = \dfrac{\langle V, N \rangle^2}{d_R}.
$$

Moreover, we have the following identities:
\begin{equation}
\begin{array}{rcl}
\dfrac{1}{d_R} &=& \textnormal{trace} \left( \hess^{\mathbb{R}^3} d_R \right) + \dfrac{\langle V, N \rangle^2}{d_R},\\
1 &=& \langle \nabla^{\mathbb{R}^3}d_R, N \rangle^2 + \langle V, N \rangle^2 + \eta^2,
\end{array}
\end{equation}
where we use in the second equation that $\eta = \langle e_3, N \rangle$. We conclude that
\begin{equation}
\textnormal{trace}\left( \hess^{\mathbb{R}^3} d_R \right) = \dfrac{\langle \nabla^{\mathbb{R}^3}d_R, N \rangle^2 + \eta^2}{d_R}. \label{tr-hessian}
\end{equation}

Now, we compute the drift Laplacian $\Delta^\varphi$ of $f$. On the one hand, the Laplacian of $f$ on $\Sigma$ is given by
$$
\begin{array}{rcl}
\Delta f &=& \textnormal{trace} \left( \hess^{\mathbb{R}^3} d_R \right) + \langle \nabla^{\mathbb{R}^3}d_R, \vec{H} \rangle,\\
&=& \dfrac{1-\vert\nabla f\vert^2 + \eta^2}{d_R} + \langle \nabla^{\mathbb{R}^3}d_R, \dot{\varphi} \eta N \rangle, \\
&=& \dfrac{1-\vert\nabla f\vert^2 + \eta^2}{d_R} + \dot{\varphi} \eta \langle \nabla^{\mathbb{R}^3}d_R,  N \rangle, \\
\end{array}
$$

On the other hand, since $\nabla f =  \nabla^{\mathbb{R}^3} d_R - \langle \nabla^{\mathbb{R}^3} d_R, N \rangle N$, we get
$$
\begin{array}{rcl}
\langle \nabla \varphi, \nabla f \rangle  &=& \dot{\varphi} \langle e_3^T, \nabla^{\mathbb{R}^3}d_R \rangle, \\
&=& - \dot{\varphi} \eta \langle N, \nabla^{\mathbb{R}^3}d_R \rangle
\end{array}
$$

Therefore,

\begin{equation}
\Delta^{\varphi} f = \dfrac{1-\vert\nabla f\vert^2 + \eta^2}{d_R} \geq \dfrac{1-\vert\nabla f\vert^2}{d_R} \label{estimate-laplacian}
\end{equation}

Let us consider the piece of $\Sigma\cap\mathcal{V}_{R}$ that is contained in $\mathcal{V}_R$ with boundary $\partial\Sigma=\Sigma\cap\partial\mathcal{V}_{R}$. We can take $R$ large enough, so that $\Sigma \cap \mathcal{V}_R \neq \varnothing$. There are two possibilities:

If $\Sigma \cap \mathcal{V}_R$ is compact, then we can apply the strong maximum principle to get a contradiction with \eqref{estimate-laplacian}. On the other hand, if $\Sigma \cap \mathcal{V}_R$ is not-compact since $f^{*}=\sup_{\Sigma} f=\sup_{\Sigma\cap\mathcal{V}_{R}}f>R,$ we can apply the Omori-Yau maximum principle with boundary \ref{OYboundary} to obtain a contradiction with \eqref{estimate-laplacian} due to the existence of a sequence of points $\{p_{n}\}_{n}\subset\Sigma$ such that 
$$f(p_{n})\rightarrow f^{*} \ \ \vert\nabla f(p_{n})\vert\rightarrow 0 \ \ \text{and} \ \ \Delta^{\varphi}f\leq\frac{1}{n} \ \ \forall n\in\mathbb{N}.$$

$\hfill\square$

Notice that if we can assure the $\varphi$-stochastically completeness of an open subset of $\Sigma\backslash\partial\Sigma$ containing the supremum $f^*$ in its interior, we can prove the following consequence for $[\varphi,\vec{e}_{3}]$-minimal surfaces with $\dot{\varphi}$ a bounded function.
\begin{corollary}
Under the assumption as above. If $\Sigma$ is locally convex in $\Sigma\cap\mathcal{V}_{R}$ with $\text{inf}_{\Sigma\cap\mathcal{V}_{R}}\vert\mathcal{S}\vert^2\geq\epsilon$, for some $\epsilon>0$, then $\Sigma$ cannot be contained in $H_{1}\cap H_{2}$.
\end{corollary}
\begin{proof}
Using the same argument as in Theorem B, we get $$f^{*}=\sup_{\Sigma} f=\sup_{\Sigma\cap\mathcal{V}_{R}}f>R$$ and the existence of a sequence $\{p_{n}\}_{n}\subset\Sigma$ such that 
$$f(p_{n})\rightarrow f^{*} \ \ \text{and} \ \ \Delta^{\varphi}f\leq\frac{1}{n} \ \ \forall n\in\mathbb{N}.$$
Notice that, from \eqref{estimate-laplacian} together with the maximum principle, these sequences $\{p_{n}\}$ must be divergent and, moreover, $d_{\Sigma}(p_{n},\partial(\Sigma\cap\mathcal{V}_{R}))>\epsilon$ uniformly for some $\epsilon>0$ since $f^{*}>R$. Now, suppose that there exists a subsequence of points $\{p_{n}\}$ (denoted in the same way) such that $\vert \nabla f\vert(p_{n})\rightarrow 1$ and $\eta^2(p_{n})\rightarrow 0$. Otherwise, we get a contradiction with \eqref{estimate-laplacian}. In the following, we assume that $\vert\nabla f\vert (p_{n})\rightarrow 1$ and $\eta^2(p_{n})\rightarrow 0$ as $n\rightarrow+\infty$. Hence, from our assumptions, the mean curvature $H(p_{n})\rightarrow 0$ as $n\rightarrow +\infty$. Since $K(p_{n})\geq 0$, each principal curvature $k_{i}(p_{n})\rightarrow0$ as $n\rightarrow+\infty$, leading to a contradiction.
\end{proof}

In general, we can relax the hypothesis on the function $\varphi$ and the principal curvatures if we assume that the strong maximum principle for $\Delta^{\varphi}$ holds in any open subset of $\Sigma$.

\begin{corollary}
\label{Cor2}
There is no $[\varphi,\vec{e}_{3}]$-minimal surface $\Sigma$ contained in two transverse vertical half-spaces of $\mathbb{R}^3$ if the Omori-Yau maximum principle for $\Delta^{\varphi}$ is valid in any open subset of $\Sigma$ with boundary.
\end{corollary}

Notice that the existence of a proper $[\varphi,\vec{e}_{3}]$-minimal surface in a domain determined by two transverse vertical planes only depends on whether the Omori-Yau maximum principle with boundary holds.

From the proof of Theorem B, Theorem \ref{OYboundary} is sharp in the sense of the quadratic growth of $\varphi$. In fact, we have already discussed in the Introduction that Theorem B is sharp on the growth of $\varphi$. Such a condition is used only for the application of Theorem \ref{OYboundary}. Therefore, if Theorem \ref{OYboundary} is true when $\varphi$ has a higher growth, then Theorem B would also be true in this case, which is a contradiction.

\section{Proof of Theorem D}

Suppose that the second statement does not occur, and argue by contradiction. We only prove the case that there exists a $\varphi$-stochastically complete $[\varphi,\vec{e}_{3}]$-minimal surface which is not a vertical plane $\Sigma$ in the half-space
$$\{x\in\mathbb{R}^{3}: \langle x,\vec{v}\rangle>0\} \ \ \text{with} \ \ \langle\vec{v},\vec{e}_{3}\rangle=0.$$
Consider the following smooth function $\psi:\Sigma\rightarrow\mathbb{R}$ given by
$$\psi(p)=\log(a+\langle p,\vec{v}\rangle) \ \ \text{for any} \ \ p\in\Sigma$$
and $a>0$ is a fixed real number. We can ensure the existence of $\psi_{*}=\text{inf}_{\Sigma}\psi$. For the other case in which $\Sigma$ is contained in $\{x\in\mathbb{R}^{3}: \langle x,\vec{v}\rangle<0\}$ the following proof also works using the function
$$\psi(p)=\log(a-\langle p,\vec{v}\rangle) \ \ \text{for any} \ \ p\in\Sigma$$

Now, fix any point $p\in\Sigma$ and consider any vector field $X\in T_{p}\Sigma$, then
\begin{equation}
\label{gradpsi}
d\psi_{p}(X)=\langle X,\nabla\psi(p)\rangle=\frac{\langle X,\vec{v}^{T}\rangle}{a+\langle p,\vec{v}\rangle}.
\end{equation}
Then, for any orthonormal frame $\{v_{i}\}_{i=1,2}$ of $T_{p}\Sigma$, we obtain
\begin{equation}
\label{eq3}
\langle D_{v_{i}}\nabla\psi,v_{i}\rangle=-\frac{\langle v_{i},v^{T}\rangle^{2}}{(a+\langle p,\vec{v}\rangle)^{2}}+\frac{1}{a+\langle p,\vec{v}\rangle}\langle D_{v_{i}}v^{T},v_{i}\rangle.
\end{equation}
Moreover, since $\langle\vec{v},\vec{e}_{3}\rangle=0$, we get that
\begin{equation}
\label{eq4}
\langle\nabla\varphi,\nabla\psi\rangle=-\frac{\dot{\varphi}}{a+\langle p,\vec{v}\rangle}\langle N,\vec{e}_{3}\rangle\langle N,\vec{v}\rangle.
\end{equation}
Consequently, from \eqref{eq3} and \eqref{eq4}, the following hold
\begin{equation}
\label{lap2}
\Delta^{\varphi}\psi=-\vert\nabla\psi\vert^{2}
\end{equation}
On the one hand, $\psi$ does not reach its infimum in an interior point of $\Sigma$. Otherwise, from \eqref{gradpsi} there exists a point $p\in\Sigma$ such that $\eta (p)=0$, but this fact is impossible. On the other hand, since $\Sigma$ is $\varphi$-stochastically complete and there exists $\psi_{*}=\text{inf}_{\Sigma}\psi$, we obtain the existence of a divergent sequence of points $\{p_{n}\}_{n\in\mathbb{N}}\subset\Sigma$ such that
$$\psi(p_{n})\rightarrow\psi_{*} \ \ \text{and} \ \ \Delta^{\varphi}\psi(p_{n})\geq-\frac{1}{n} \ \ \text{for any} \ \ n\in\mathbb{N}.$$
However, from \eqref{lap2}, we get $\vert\nabla\psi\vert (p_ {n})\vert \rightarrow 0$ as $n\rightarrow+\infty$ because
$$-\frac{1}{n}\leq\Delta^{\varphi}\psi(p_{n})=-\vert\nabla\psi\vert^{2}(p_{n})\leq 0 \ \ \text{for any} \ \ n\in\mathbb{N}.$$
Notice that $(a+\langle p_ {n},\vec{v}\rangle)$ is bounded for any $n\in\mathbb{N}$ since $\psi(p_{n})\rightarrow \psi_{*}$. Thus, $v^{T}(p_{n})\rightarrow 0$. But in this case, up to subsequence, we find a sequence of points $\{p_ {n}\}\subset\Sigma$ such that $\eta(p_{n})\rightarrow 0$, which leads to a contradiction.
$\hfill\square$

\begin{corollary}
\label{Cor1}
If $\Sigma$ is a $[\varphi,\vec{e}_{3}]$-minimal surface satisfying the generalized Omori-Yau maximum principle for $\Delta^{\varphi}$ which is contained in the half-space $\mathcal{H}_{\vec{v}}=\{x\in\mathbb{R}^{3}: \langle x,\vec{v}\rangle\leq 0\}$ for any vector $\vec{v}\in\mathbb{R}^3$ such that $\langle\vec{v},\vec{e}_{3}\rangle=0.$ then one of the following holds
\begin{enumerate}
\item $\Sigma$ is a plane parallel to $\partial \mathcal{H}_{\vec{v}}$.
\item $\Sigma$ is not a plane parallel to $\partial \mathcal{H}_{\vec{v}}$, but there exists a divergent sequence $\{p_{n}\}_{n\in\mathbb{N}}\subset\Sigma$ such that $\eta(p_{n})\rightarrow 0$ as $n\rightarrow+\infty$.
\end{enumerate}
In particular, if $\Sigma$ is a proper $[\varphi,\vec{e}_{3}]$-minimal surface in $\mathbb{R}^3$ whose function $\varphi$ satisfies the condition \eqref{condK} , then the previous statement holds. 
\end{corollary}

\begin{corollary}
\label{corEk}
 Let $\Sigma$ be a complete locally vertical $[\varphi,\vec{e}_{3}]$-minimal graph such that the principal curvatures satisfy $k_{i}>\xi>0$, with $i=1,2$,outside a compact set of $\Sigma$. If $\Sigma$ is $\varphi$-stochastically complete, then $\Sigma$ cannot be contained in $\mathcal{H}_{\vec{v}}$ with $\langle\vec{v},\vec{e}_{3}\rangle=0$.
\end{corollary}
\begin{proof}
Argue by contradiction. Suppose that $\Sigma$ is contained in the closed half-space $\mathcal{H}_{\vec{v}}$ with $\langle\vec{v},\vec{e}_{3}\rangle=0$ under the same conditions. Assume that $\eta\geq 0$ (the same argument works if $\eta\leq 0$). Notice that $\Sigma$ is not a vertical plane since $K>0$. Then, there exists a sequence of points $\{p_{n}\}_{n\in\mathbb{N}}\subset\Sigma$ such that $\eta(p_{n})\rightarrow 0=\text{inf}_{\Sigma}\eta$. From Proposition \ref{ekeland} together with the completeness, there exists a sequence $\{q_{n}\}_{n\in\mathbb{N}}\subset\Sigma$ such that $\vert\nabla\eta(q_{n})\vert=\vert\textbf{S}\vec{e}_{3}^{T}(q_{n})\vert\rightarrow 0$, where $\textbf{S}$ is the shape operator of $\Sigma$. Consequently, up to subsequence, one of the principal curvatures $k_{i}$ satisfies $k_{i}(q_{n})\rightarrow 0$ as $n\rightarrow +\infty$ since $\vert\nabla\mu\vert(q_{n})\rightarrow 1$, getting to a contradiction.
\end{proof}
\begin{remark}
The convexity of the Gauss curvature is a natural property in this family of surfaces. For example, see the work of J. Spruck and L. Xiao \cite{SX} where they showed the convexity of the complete mean convex translator in $\mathbb{R}^3$. On the other hand, we proved the convexity for a family of complete stable $[\varphi,\vec{e}_{3}]$-minimal surfaces in \cite{MMJ}.
\end{remark}
Finally, as a nice application of Theorems C and D, we prove a strong half-space result for complete properly embedded $[\varphi,\vec{e}_{3}]$-minimal surfaces in $\mathbb{R}^{3}$. Namely,

\begin{theorem}[Strong half-space Theorem] Let $\Sigma_{1}$ and $\Sigma_{2}$ be complete mean convex properly embedded $[\varphi,\vec{e}_{3}]$-minimal surfaces in $\mathbb{R}^{3}$ and whose function $\varphi$ is a diffeomorphism such that $\vert\dot\varphi\vert>\xi$ for a real constant $\xi>0$ outside a compact set of $\Sigma$. If $\Sigma_{1}$ and $\Sigma_{2}$ are $\varphi$-stochastically complete, then one of the following holds:
\begin{enumerate}
\item  $\Sigma_{1}$ and $\Sigma_{2}$ intersect non-trivially,
\item $\Sigma_1$ and $\Sigma_2$ are contained in distinct half-spaces determined by a vertical plane. Moreover, on each $\Sigma_{i}$, there exists a sequence of points such that one of the main curvatures converges to zero along the sequence.
\end{enumerate}
\end{theorem}
\begin{proof}
Suppose that case $1$ does not occur. Let $N$ be the geodesic closure  of the component of $\mathbb{R}^{3}\backslash(\Sigma_{1}\cup\Sigma_{2})$ that has boundary $\partial N=\Sigma_{1}\cup\Sigma_{2}$ not connected.  {From the assumption}, the mean curvature of $\partial N$ is always nonnegative with respect to the outward normal. Moreover, $N$ is Ricci-flat  and thus, from \cite[Theorem 8]{MSY}, $N$ contains a plane $P$ since it is the only complete stable minimal surface in $\mathbb{R}^{3}$, see \cite{FCS}. Applying Theorem C, the only possibility for $P$ is a vertical plane. Now, applying theorem E, there exist two sequences of points $\{p_{n}\}\subset\Sigma_{1}$ and $\{q_ {m}\}\subset\Sigma_{2}$ such that $\eta_{\Sigma_{1}}(p_{n})\rightarrow 0$ and $\eta_{\Sigma_{2}}(q_{m})\rightarrow 0$ as $n,m\rightarrow +\infty$. Since each $\Sigma_{i}$ is mean convex and $\dot{\varphi}\neq 0$, each angle function reaches either its supremum or its infimum. Hence, the proof follows by arguing in the same way that the Corollary \ref{corEk} together with Proposition \ref{ekeland}.
\end{proof}

\section{Proof of Theorem E}

We follow similar ideas as given in \cite{KP}. Namely, suppose by contradiction that $\Sigma$ is contained in the cone $C_{\vec{e}_3,a}$. Consider the function $ \widetilde{\psi} : \Sigma \rightarrow \mathbb{R}$, given by $\widetilde{\psi}(x)=\mu(x) - a |x|$, where $\mu(x)=\langle x,e_{3}\rangle$ and $\{e_{i}\}_{i=1,2,3}$ is the canonical frame of $\mathbb{R}^3$.  Notice that $\widetilde{\psi}$ is smooth on $\Sigma$. Therefore, we have:
\begin{eqnarray}
\vert\nabla \widetilde{\psi}\vert&=&\vert e_3^{T} - a \nabla \vert x\vert\vert =\left| e_3^{T} - a \dfrac{x^{T}}{\vert x\vert}\right| \label{grad-thmC} \\
\Delta \widetilde{\psi} &=& \dot{\varphi} \eta^2 - \dfrac{a}{\vert x\vert} \left( \dot{\varphi} \eta \langle x , N \rangle + 2 - \dfrac{\vert x^{T}\vert^2}{\vert x\vert^2} \right) \label{lap-thmC} \\
\langle \nabla \varphi, \nabla \widetilde{\psi} \rangle &=& \dot{\varphi} \langle e_3^T, e_3^T\rangle - a \dfrac{x^T}{\vert x\vert}
\end{eqnarray}
Hence,
$$
\begin{array}{rcl}
\Delta^{\varphi} \widetilde{\psi} &=& - \dfrac{2a}{\vert x\vert} + a \dfrac{\vert x^T\vert^2}{\vert x\vert^3} + \dot{\varphi} - \dfrac{a \dot{\varphi}}{\vert x\vert} (\psi(x)+a\vert x\vert) \\
&\geq& - \dfrac{2a}{\vert x\vert} + \dot{\varphi} - \dfrac{a \dot{\varphi}\psi(x)}{\vert x\vert} - a^2 \dot{\varphi}
\end{array}
$$
and since $\widetilde{\psi}\leq 0$ by the hypothesis of contradiction, we have
\begin{align}
\Delta^{\varphi} \widetilde{\psi} &> - \dfrac{2a}{\vert x\vert} + \dot{\varphi}(1 - a^2) \nonumber\\
&= (1-a^2) \left( \dot{\varphi} - \dfrac{2a}{(1-a^2)\vert x\vert} \right). \label{inqc}
\end{align}
Now, notice that there exists $\psi^{*}=\text{sup}_{\Sigma}\widetilde{\psi}$ and then there exists a sequence of points $\{p_{n}\}_{n\in\mathbb{N}}\subset\Sigma$ such that
\begin{equation}
\label{estc}
\widetilde{\psi}(p_{n})\longrightarrow\psi^{*} \ \ \text{and} \ \ \frac{1}{n}\geq\Delta^{\varphi}\widetilde{\psi} \ \ \text{for any} \ \ n\in\mathbb{N}.
\end{equation} 
Up to subsequence (denoted in the same way), we can assume that the sign of $\mu(p_{n})$ is constant. Firstly, suppose that $\mu(p_{n})\geq 0$ for any $n\in\mathbb{N}$. Since $\widetilde{\psi} \leq 0$, we have $\mu(x) \leq a \vert x\vert$. Consequently, along with this sequence, we get
$$ - \dfrac{2a}{(1-a^2)\vert x\vert} \geq - \dfrac{2a}{(1-a^2)\mu(x)}$$ 
and
$$
\dot{\varphi}(\mu(x)) - \dfrac{2a}{(1-a^2)\vert x\vert} \geq \dot{\varphi}(\mu(x)) - \dfrac{2a}{(1-a^2)\mu(x)}.
$$
Before we continue, let us prove that we can take finite vertical translations of $\Sigma$. Let $\lambda>0$, and for the surface $\Sigma - \lambda e_3$, consider its corresponding mean curvature $H_\lambda$ and angle function $\eta_\lambda$. We know that
$$
\begin{array}{rcl}
H(p) &=& H_{\lambda} (p - \lambda e_3), \\
\eta (p) &=& \eta_{\lambda} (p - \lambda e_3).
\end{array}
$$
Furthermore, $\mu_{\lambda} (p - \lambda e_3) = \langle p - \lambda e_3, e_3 \rangle = \mu(p) - \lambda$. Since $H(p) = - \dot{\varphi}(\mu(p)) \eta (p)$, we have
$$H_{\lambda} (p - \lambda e_3) =- \dot{\varphi}(\mu_{\lambda} (p - \lambda e_3)+\lambda) \eta_{\lambda} (p-\lambda e_3).$$
Define $\varphi_{\lambda}(x_3):=\varphi(x_3 + \lambda)$ so that $\dot{\varphi_{\lambda}}(x_3)=\dot{\varphi}(x_3+\lambda)$. Therefore, 
$$\dot{\varphi} (\mu_{\lambda} (p+\lambda e_3)+\lambda)=\dot{\varphi_{\lambda}}(\mu_{\lambda}(p-\lambda e_3))$$
and we conclude that
$$H_{\lambda} (p - \lambda e_3) =- \dot{\varphi_{\lambda}}(\mu_{\lambda}(p-\lambda e_3)) \eta_{\lambda} (p-\lambda e_3).$$
Therefore, $\Sigma - \lambda e_3$ is a $[\varphi_\lambda, e_3]$-minimal surface, as we wanted.

Suppose that $\mu(p_{n})\longrightarrow\mu$ as $n\rightarrow+\infty$.  Up to vertical and horizontal translation, we can assume that $\mu>1$. Otherwise, if $\mu(p_{n})$ is divergent, then we get to a contradiction with \eqref{inqc}. Then, from \eqref{inqc} and \eqref{estc} together our assumptions, we get that
$$\frac{1}{n}\geq\Delta^{\varphi}\widetilde{\psi(}p_{n})\geq (1-a^2)\left( \xi-\frac{2a}{1-a^2}\right)>0.$$
Hence, we can ensure the existence of the limit of $\Delta^{\varphi}\widetilde{\psi} (p_{n})$, which leads to a contradiction.

On the other hand, suppose that $\mu(p_{n})\leq 0$ for any $n\in\mathbb{N}$. In this case, we will assume that $\mu(p_{n})\longrightarrow-\infty$ as $n\rightarrow+\infty$. Otherwise, if up to subsequence, $\mu(p_{n})$ has limit $\mu\leq 0$. Then, we can take a translated surface as before, where $\mu > 1$, and the previous reasoning applies. On the other hand, if $\mu(p_{n})\longrightarrow-\infty$, then $| p_{n}|\longrightarrow+\infty$ as $n\rightarrow+\infty$. Hence, taking the limit in \eqref{inqc}, we get that
$0\geq \xi>0,$
getting again a contradiction.
$\hfill\square$

\section{Concluding remarks}
We finish this work with the following related questions:
\begin{itemize}
\item Motivated by Theorems C, D, and the strong half-space theorem, we ask under what assumptions for $\varphi$ the only complete $[\varphi,\vec{e}_{3}]$-minimal surface contained in a half-space is the vertical plane.

\item Motivated by the work of N. Nadirashvili \cite{NN}, an interesting question would be to find constraints for $\varphi$ or conditions over the curvature to ensure the existence of complete $[\varphi,\vec{e}_{3}]$-minimal surfaces in a compact convex domain of $\mathbb{R}^3$.

\item Motivated by the works of Colding and Minicozzi \cite{CM2}, Meeks, P\'erez and Ros \cite{MPR}, and Meeks and Tinaglia \cite{MT}, we ask what conditions for $\Sigma$ and function $\varphi$ would imply properness in $\mathbb{R}^3.$

\end{itemize}

\section{Acknowledgments}

A. L. Martinez-Trivi\~no was partially supported by Ministerio de Ciencia e Innovaci\'on Grants No: PID2020-118137GB-I00/AEI/10.13039/501100011033, PID2021-126217NB-100, FQM-398: (GELORI) Lorentzian and Riemannian geometry and the ``Maria de Maeztu'' Excellence Unit IMAG, reference CEX2020-001105-M, funded by MCIN/AEI/10.13039/501100011033. J. P. dos Santos was partially supported FAPDF - Funda\c{c}\~ao de Apoio a Pesquisa do Distrito Federal, grant number 00193-00001678/2024-39, CNPq - Conselho Nacional de Desenvolvimento Cient\'ifico e Tecnol\'ogico, Brazil, grant number 402589/2022-0, and he is grateful to the Mathematics Department of King's College London for the hospitality, where part of this work was conducted.

\end{document}